\documentclass[11pt,a4paper]{amsart}
\usepackage{amscd}
\usepackage{amssymb}
\usepackage[centering,text={15.5cm,22cm}]{geometry}
\usepackage{graphicx,color}
\usepackage[all]{xy}
\usepackage{mathrsfs}
\usepackage{marvosym}
\usepackage{stmaryrd}
\usepackage{srcltx}
\usepackage{hyperref}
\usepackage{marginnote}
\usepackage{upgreek} 
\usepackage{bm}

\usepackage{relsize}

\definecolor{shadecolor}{rgb}{1,0.9,0.7}

\newtheorem{theorem}{Theorem}[section]
\newtheorem{lemma}[theorem]{Lemma}
\newtheorem{lemma-definition}[theorem]{Lemma-Definition}
\newtheorem{proposition}[theorem]{Proposition}
\newtheorem{corollary}[theorem]{Corollary}
\newtheorem{conjecture}[theorem]{Conjecture}

\theoremstyle{definition}

\theoremstyle{remark}
\newtheorem{remark}[theorem]{Remark}

\numberwithin{equation}{section}
\numberwithin{figure}{section}

\newcommand{\ZZ} {\mathbb{Z}}

\newcommand{\PP} {\mathbb{P}}
\renewcommand{\AA} {\mathbb{A}}

\newcommand {\shA}  {\mathcal{A}}

\newcommand {\shC}  {\mathcal{C}}
\newcommand {\shD}  {\mathcal{D}}

\newcommand {\shE}  {\mathcal{E}}

\newcommand {\shL}  {\mathcal{L}}
\newcommand {\shM}  {\mathcal{M}}

\newcommand {\shN}  {\mathcal{N}}
\newcommand {\shO}  {\mathcal{O}}

\newcommand {\shT}  {\mathcal{T}}

\newcommand {\shX}  {\mathcal{X}}

\newcommand {\foC}  {\mathfrak{C}}

\newcommand {\foM}  {\mathfrak{M}}

\newcommand {\Bl}  {\operatorname{Bl}}

\newcommand {\diag} {\operatorname{diag}}

\newcommand {\ev}  {\operatorname{ev}}

\newcommand {\Hom}  {\operatorname{Hom}}

\newcommand {\id}  {\operatorname{id}}

\newcommand {\la}  {\leftarrow}

\newcommand {\lla}  {\longleftarrow}

\newcommand {\lcm}  {\operatorname{lcm}}

\newcommand {\lra}  {\longrightarrow}

\newcommand {\oM}  {\overline{\mathcal{M}}}

\newcommand {\pr}  {\operatorname{pr}}

\newcommand {\pt} {\operatorname{pt}}

\newcommand {\ra}  {\to}

\newcommand {\rk} {\operatorname{rk}}

\newcommand {\Tot}  {\operatorname{Tot}}

\DeclareMathOperator {\GW} {GW}

\DeclareMathOperator {\vdim} {vdim}

\def\mydate{\ifcase\month \or January\or February\or March\or
April\or May\or June\or July\or August\or September\or October\or 
November\or December\fi \space\number\day,\space\number\year}

\newcommand{\htd}{{H_2(D)^+}}

\newcommand{\htz}{{H_2(Z)^+}}
\newcommand{\vir}{{\rm vir}}

\begin{document}


\title
[Local Gromov-Witten Invariants are Log Invariants]
{Local Gromov-Witten Invariants are Log Invariants}
\author{Michel van Garrel, Tom Graber, Helge Ruddat}

\address{\tiny University of Warwick, Mathematics Institute, Coventry CV4 7AL, UK}
\email{michel.van-garrel@warwick.ac.uk}

\address{\tiny Caltech, Department of Mathematics, MC 253-37, 362 Sloan Laboratory, Pasadena, CA 91125, USA}
\email{graber@caltech.edu}

\address{\tiny JGU Mainz, Institut f\"ur Mathematik, Staudingerweg 9, 55128 Mainz, Germany}
\email{ruddat@uni-mainz.de}
\thanks{This work was supported by HR's DFG Emmy-Noether grant RU1629/4-1 and by MvG's affiliation with the Korea Institute for Advanced Study.}

\maketitle
\setcounter{tocdepth}{1}
\tableofcontents
\bigskip

\section{Introduction}

Let $X$ be a smooth projective variety, $D$ a smooth nef divisor on $X$, and $\beta$ a curve class on $X$ with $d:=\beta\cdot D >0$.  The goal of this note
is to prove a simple equivalence between two virtual counts of rational curves on $X$ that can be associated to this situation.  First there are the {\em local invariants},
the Gromov-Witten invariants which virtually count curves in the total space of $\shO_X(-D)$.  These can be defined as integrals against the virtual fundamental class
of $\oM_{0,0}(\Tot (\shO_X(-D)))$.  The conditions on $D$ and $\beta$ are exactly the ones under which these are well-defined.  Second, we can consider the relative or logarithmic invariants which virtually count rational curves in $X$ which intersect $D$ in a 
single point, necessarily with multiplicity $d$.  These could also be thought of as a virtual count of maps from $\AA^1$ to the open variety $X\backslash D$.  
There exist several moduli spaces that can be used to define these counts.  We will use the space of logarithmic stable maps to the log smooth space $X(\log D)$ associated to the pair $X$ and $D$.  This space parametrizes many types of maps with different specified contact orders, but the one directly relevant to this problem can be denoted $\oM_{0,(d)}(X(\log D),\beta)$, where the $(d)$ is meant to denote a single marked point having maximal contact order with $D$.  Our main result is a comparison of the two virtual fundamental classes here.  

\begin{theorem} \label{mainthm-intro}
$F_*[\oM_{0,(d)}(X(\log D),\beta)]^\vir=(-1)^{d+1}d[\oM_{0,0}(\Tot(\shO_X(-D)),\beta)]^\vir$.
\end{theorem}
Here $F$ denotes the map that takes a logarithmic map to $X(\log D)$ and forgets the logarithmic structure as well as the marked point.  By the negativity of $\shO_X(-D)$, the space of maps to $\shO_X(-D)$ is just the same as the space of maps to $X$, so both sides of the formula are being considered as elements of $A_*(\oM_{0,0}(X,\beta))$.   We remark that we show the same formula holds if we add $n$ marked points to the moduli problems on each side (with 0 contact order with $D$ in the case of left hand side). This equality of virtual cycles leads to analogous equalities of certain numerical invariants after capping with constraints and integrating.

Precisely, let $\gamma=(\gamma_1,....,\gamma_n)$ denote a collection of insertions, with each $\gamma_i$ either a primary invariant (i.e. no descendants) or a descendant of a cycle class which does not meet $D$.  
Let $N_\beta(\gamma)$ denote the genus zero local Gromov-Witten invariant for these insertions and $R_\beta(\gamma)$ the genus zero relative Gromov-Witten invariant with the same insertions (and one maximal contact order marking as before), then as a direct corollary of Theorem~\ref{mainthm-intro}, we find:
\begin{corollary} \label{maincor-intro}
$R_\beta(\gamma)=(-1)^{d+1}dN_\beta(\gamma)$.
\end{corollary}
This follows from Theorem~\ref{mainthm-intro} and the projection formula, since the restriction on the insertions corresponds exactly to the condition that the constraints on the space of log stable maps be pulled back from classes on $\oM_{0,n}(X,\beta)$.

This formula was first conjectured by Takahashi \cite{Ta01} for $X=\PP^2$ and $D$ a smooth cubic.  It was proven in this case by Gathmann via explicit calculation of both sides \cite{Ga03}.  In an unpublished note, Graber and Hassett gave a simple proof of the formula when $X$ is any smooth del Pezzo surface and $D$ is a smooth anti-canonical divisor.  The proof we will give here follows the main idea of that argument which is to apply the degeneration formula for Gromov-Witten invariants \cite{LR01, Li02, IP04, AF11, Ch14, KLR17} to a twist of the degeneration to the normal cone of the preimage of $D$ in $\Tot(\shO(-D))$.

In the general setting though, we need to take into account the existence of rational curves in $D$.  In order to control these, we need to study the following situation which may be of independent interest, comparing the virtual geometry of genus zero stable log maps to certain projective bundles to the geometry of stable maps to the base.

Let $\shA$ be a nef line bundle on a smooth projective variety $D$, 
$$Y=\mathbb{P}(\shO_D\oplus\shA) \stackrel{p}{\longrightarrow} D$$ 
the $\PP^1$-bundle over $D$ and $\beta$ a curve class on $Y$.
Let $D_0$ be the section of $p$ with normal bundle $\shA^\vee$ and let $\oM_{0,\Gamma}(Y(\log D_0),\beta)$ be the space of genus zero basic stable log maps to $Y$ where $\Gamma$ is a list of $n$ points with prescribed contact orders to $D_0$. We have a pushforward map 
$$w: \oM_{0,\Gamma}(Y(\log D_0),\beta) \to \oM_{0,n}(D,p_*\beta)$$
which forgets the log structures, composes with the projection to $D$ and stabilizes.  We would like to say that the virtual class on the space of log stable maps to $Y(\log D_0)$ is the pullback of the virtual class of the space of stable maps to $D$, but since the map $w$ need not be flat, there is no well-defined pullback in general.  Nevertheless, we have the following result.

\begin{theorem}[Theorem \ref{thm-pullback} below] \label{thm-P1-bundle-pullback} 
The map $w$ factors through an intermediate space $\shM$, i.e. 
$\oM_{0,\Gamma}(Y(\log D_0),\beta)\stackrel{u}{\lra}\shM\stackrel{v}{\lra} \oM_{0,n}(D,p_*\beta)$ where $u$ is smooth and $v$ a pull-back of a map $\nu$ to a smooth stack, see \eqref{main-diagram-Y-D} for details. Then
$$ [\oM_{0,\Gamma}(Y(\log D_0),\beta)]^\vir = u^*\nu^![\oM_{0,n}(D,p_*\beta)]^\vir. $$
\end{theorem}

In fact, this theorem is proven in more generality below -- the only feature needed of the map
$p:Y(\log D_0) \to D$ is that it is a log smooth projective morphism whose relative log tangent bundle satisfies a positivity condition.


We believe that our main result generalizes as follows.
\begin{conjecture} \label{conjecture} 
Let $D\subset X$ be a normal crossing divisor with smooth nef components $D_1,...,D_k$ and $\beta$ a curve class such that $d_i:=\beta\cdot D_i>0$ for all $i$. Let
$$F:\oM_{0,(d_1),...,(d_k)}(X(\log D),\beta)\ra \oM_{0,0}(X,\beta)$$ 
be the forgetful map from the moduli space of genus zero basic stable log maps with one marking for each $D_i$ requiring maximal order of contact $d_i$ at $D_i$, then
$$F_*[\oM_{0,(d_1),...,(d_k)}(X(\log D))]^\vir=\left(\prod_{i=1}^k(-1)^{d_i+1}d_i\right)[\oM_{0,0}(\Tot(\shO_X(-D_1)\oplus....\oplus \shO_X(-D_k)))]^\vir.$$
\end{conjecture}
We conjecture also the similar statement where further markings are added on both sides (with zero contact orders to the $D_i$ for the left hand side).
As evidence for this conjecture, note that it allows computing the local invariant $\frac{1}{d^3}$ of $\shO_{\PP^1}(-1,-1)$ for degree $d$ curves from the unique Hurwitz cover of $\PP^1$ with maximal branching at $0$ and $\infty$ and cyclic order $d$ automorphism, so the log invariant is $\frac1{d}$, see also \cite[Remark 4.17]{MR16}.
We also checked Conjecture~\ref{conjecture} for $\PP^n$ with $D$ the toric boundary in the case of a single insertion of $\psi^{n-1}[pt]$: the log invariant can easily be computed from \cite[Theorem 1.1]{MR16}+\cite{MR19} and equals $1$ for all $n$ and $d$.
We are grateful to Andrea Brini for computing for us the local invariant for this situation confirming the conjecture in this case.
We thank Dhruv Ranganathan for pointing out that \cite[Lemma~3.1]{PZ08} computes the virtual count of degree $d$ genus zero curves in $\Tot(\shO_{\PP^2}(-1)^{\oplus 3})$ passing through two points as $(-1)^{d-1}/d$ and using \cite[Theorem 1.1]{MR16} for the computation of the log invariant confirms the conjecture also in this case. Further evidence will appear in \cite{BBvG}.

Once the technology of punctured Gromov-Witten invariants and a more general degeneration formula is fully developed, we expect that the proof we give in this paper could be iterated to imply Conjecture~\ref{conjecture}.  If the $D_i$ are disjoint (as in the first example above) the existing
technology is already enough to establish the conjecture.

We are grateful to Dan Abramovich, Mark Gross, Davesh Maulik, Rahul Pandharipande and James Pascaleff for helpful and inspiring discussions, and to the anonymous referee for valuable suggestions.


\section{Deducing the main result from the degeneration formula}
\label{section-setup}

First we briefly describe the strategy of proof.  We have a smooth projective variety $X$ containing a nef divisor $D$.  If we let $\shX = \Bl_{D\times\{0\}}(X\times \AA^1)\ra \AA^1$ be the degeneration to the normal cone of $D$ in $X$, then we get a family
over $\AA^1$ whose general fiber is $X$ and whose special fiber is a union of a copy of $X$, which we denote by $X_0$, and a $\PP^1$-bundle over $D$, which we will denote by $Y$.  Precisely, if we let 
$\shA=\shN_{D/X}$, then $Y=\mathbb{P}(\shO_D\oplus\shA)\stackrel{p}{\ra} D$.  This $\PP^1$ bundle comes with two obvious sections whose images we denote by $D_0$ and $D_\infty$ which have normal bundles $\shA^\vee$ and $\shA$ respectively.  The intersection of $X_0$ with $Y$ is given by $D$ in $X_0$ and by $D_0$ in $Y$.

 To construct a degeneration of the total space of $\shO_X(-D)$, we can just look at the total space of any line bundle over $\shX$ whose restriction to a general fiber is $\shO(-D)$.  Our choice, which we denote by $\shL$, is $\Tot(\shO(-\shD))$ where $\shD$ is the proper transform of $D \times \AA^1$.  Then, since 
$\shD \cap X_0 = \emptyset$ and $\shD \cap Y = D_\infty$ we find that the projection $\shL \to \AA^1$ gives us a family whose general fiber is $\Tot\shO_X(-D)$ and whose special fiber is a union of two components $L_X$, and $L_Y$ where $L_X \cong X \times \AA^1$ and $L_Y \cong \Tot(\shO_Y(-D_\infty)$.

Applying the degeneration formula to this family will then relate Gromov-Witten theory of $\Tot (\shO_X(-D)$ to the relative or log invariants of the two pairs $(X \times \AA^1, D\times \AA^1)$ and
$(L_Y, D_0 \times \AA^1)$.  While the degeneration formula in general is quite complicated and given by a sum over combinatorial types of curve degenerations, we will see that for this degeneration in genus zero, every term but one in that sum vanishes, and we can find explicitly the contribution from $L_Y$, so we are left with an expression for the genus zero Gromov-Witten theory of $\shO_X(-D)$ in terms of that of $(X\times \AA^1, D\times \AA^1)$ which is just the same as that of $(X,D)$.

In order to carefully write down a proof of that vanishing, we will need to establish our conventions about notations for the degeneration formula.  Because we will eventually need to make use of the theory of logarithmic stable maps, we will state the version in that setting, although it is worth noting that thanks to the comparison theorems of \cite{AMW14}, we could equally well use the better known formalism for degeneration in terms of relative invariants.


\subsection{Degeneration formula} \label{sec-degenformula}
We recall the degeneration formula where for us the version in \cite[Theorem 1.4]{KLR17} is most convenient, further details on this subsection can be found there. 
The input is a space $\shX_0$ that is log smooth over the standard log point with the additional assumption that $\shX_0= X\sqcup_D Y$ is a union of two smooth components that meet along $D$, a smooth divisor in each component.
We are interested in basic stable log maps to $\shX_0$. The domain curve components map into either $X$ or $Y$ (or both) and the degeneration formula uses this fact effectively to decompose the moduli space of stable maps and hence the virtual fundamental class and Gromov-Witten invariants as a sum of contributions from each type of stable map. We next give the details for the situation of genus zero with $n\ge 0$ markings.
We recall \cite[\S2]{KLR17}: let a curve class $\beta$ in $\shX_0$ and an integer $n$ be given, and let $\Omega(\shX_0)$ denote the set of graphs $\Gamma$ with the following decorations and properties.
The vertices of $\Gamma$ are partitioned into two sets indexed by $X$ and $Y$ and no edge has vertices labeled by only $X$ or only $Y$, this property for a graph is typically called \emph{bipartite}. We will henceforth speak of $X$- and $Y$-vertices, referring to the partition membership. The edges of $\Gamma$ are enumerated $e_1,...,e_r$ (where $r$ may vary), each edge $e$ is decorated with a positive integer $w_e$, each vertex $V$ is decorated with a set $n_V$ and a class $\beta_V$ that is an effective curve class in $X$ or $Y$ depending on the bipartition type of $V$.
The set $n_V$ is a subset of $\{1,...,n\}$ to be thought of as the set of marking labels attached to $V$.
Every $\Gamma\in \Omega(\shX_0)$ is subject to the following stability condition: if $\beta_V=0$, then the valency of $V$ is at least $3$. 
Furthermore, $\beta =\sum_V\beta_V$ and $\beta_V\cdot D=\sum_{V\in e} w_e$ and $\{1,...,n\}=\coprod_V n_V$. 
These conditions make $\Omega(\shX_0)$ a finite set.

Given $\Gamma\in\Omega(\shX_0)$ and a vertex $V$ of $\Gamma$, we define $\Gamma_V$ as the ``subgraph'' of $\Gamma$ at the vertex $V$. That is, $\Gamma_V$ has a single vertex $V$ that is decorated with the set $n_V$ and curve class $\beta_V$ and has as adjacent half-edges the edges adjacent to $V$ in $\Gamma$ with their weights.
If $V$ is an $X$-vertex, we define $\oM_V:=\oM_{\Gamma_V}(X(\log D),\beta_V)$, that is, the moduli space of genus zero stable maps to $X(\log D)$ of class $\beta_V$, with edges of $\Gamma_V$ indexing the markings with contact order to $D$ given by the weight of an edge and $n_V$ enumerating additional markings (with contact order zero to $D$). 
Analogously, if $V$ is a $Y$-vertex, we set $\oM_V:=\oM_{\Gamma_V}(Y(\log D),\beta_V)$.
We define $\bigodot _V \oM_V $ by the Cartesian square
\begin{equation*} \label{gluediag}
\vcenter{ \xymatrix{ 
\bigodot _V \oM_V \ar[r]\ar_\ev[d]& \prod_V \oM_V\ar^\ev[d]\\
\prod_{e} D \ar^(.4)\Delta[r]& \prod_V\prod_{e\ni V} D.\\
} }
\end{equation*}
For each $\Gamma\in\Omega(\shX_0)$, consider the moduli space $\oM_\Gamma$ of basic stable log maps to $\shX_0$ where the curves are marked by $\Gamma$, i.e. a subset of the nodes is marked by $e_1,...,e_r$ and the dual intersection graph collapses to $\Gamma$, for details see \cite[\S4]{KLR17}. There is an \'etale map that partially forgets log structure $\Phi:\oM_\Gamma\ra \bigodot _V \oM_V$ and another finite map that forgets the graph-marking $G:\oM_\Gamma\ra\oM_{0,n}(\shX_0,\beta)$ where the latter refers to the moduli space of $n$-marked basic stable log maps to the log space $\shX_0$ that is log smooth over the standard log point.

\begin{theorem}[Degeneration formula in genus zero] \label{degen-formula}
We have
$$[\oM_{0,n}(\shX_0,\beta)]^\vir = \sum_{\Gamma\in \Omega(\shX_0)} \frac{\lcm(w_{e_1},...,w_{e_r})}{r!} G_*\Phi^*\Delta^! \prod_V[\oM_V]^\vir.$$
\end{theorem}

\subsection{Setup}
\label{sec-degen-formula}

Instead of writing $\oM_{0,n}(..)$ to refer to moduli spaces of genus zero basic stable log maps with no markings, we simply write $\oM(..)$ in the following.
Since $\shL\ra\AA^1$ and $\shX\ra\AA^1$ are log smooth when given the divisorial log structure from the central fiber respectively,
by \cite[Theorem 0.2 and Theorem 0.3]{GS13}, \cite[Theorem 1.2.1]{Ch11}, we obtain moduli spaces of basic stable log maps $\oM(\shX/\AA^1,\beta)$ and $\oM(\shL/\AA^1,\beta)$ that are proper over $\AA^1$ and whose formation commutes with base change. 
These carry virtual fundamental classes
$[\oM(\shX/\AA^1,\beta)]^\vir,[\oM(\shL/\AA^1,\beta)]^\vir$ that are compatible with base change. Moreover, since $\beta\cdot c_1(\shO_\shX(-\shD))<0$, we get that $\oM(\shX/\AA^1,\beta)=\oM(\shL/\AA^1,\beta)$ and $\oM(\shX_0,\beta)=\oM(\shL_0,\beta)$.
We furthermore consider the projection $\tilde p:\shX_0=Y\sqcup_D X_0\ra X$ induced by $p$ via the universal property of the co-product.

\begin{lemma} \label{lemma-Tot-is-p-pushforward}
Let $P:\oM(\shL_0,\beta)\ra \oM(X,\beta)$ be the map that takes a basic stable log map to $\shL_0$, forgets the log structure, composes with $\tilde p$ and stabilizes. Then
$$[\oM(\Tot(\shO_X(-D)),\beta)]^\vir = P_*[\oM(\shL_0,\beta)]^\vir.$$
\end{lemma}

\begin{proof} 
Consider the four Cartesian squares of proper morphisms
\begin{equation} \label{eq-glue-square}
\xymatrix@C=30pt
{
\oM(\shL_0,\beta)\ar@{^{(}->}[r]\ar_P[d] & \oM(\shL/\AA^1,\beta)\ar_Q[d] & \ar@{_{(}->}[l]\oM(\Tot(\shO_X(-D)),\beta)\ar@{=}[d] \\
\oM(X,\beta)\ar@{^{(}->}[r]\ar[d] & \oM(X\times\AA^1/\AA^1,\beta)\ar^p[d] & \ar@{_{(}->}[l]\oM(X,\beta)\ar[d] \\
\{0\}\ar^{i_0}@{^{(}->}[r]&\AA^1&\ar_{i_1}@{_{(}->}[l]\{1\}
}
\end{equation}
and notice that $[\oM(\shL_0,\beta)]^\vir$ is a class in the top left corner that is the Gysin pullback of $[\oM(\shL/\AA^1,\beta)]^\vir$ from the top middle which in turn Gysin pulls back to \linebreak $[\oM(\Tot(\shO_X(-D)),\beta)]^\vir$ in the top right.
The statement now follows from the commuting of Gysin pullbacks with proper pushforward applied to the top two squares, since we see that 
$$P_*[\oM(\shL_0,\beta)]^\vir = i_0^!Q_*  \oM(\shL/\AA^1,\beta) = i_1^!Q_*  \oM(\shL/\AA^1,\beta) =  [\oM(\Tot(\shO_X(-D)),\beta)]^\vir$$
where the middle equality follows since $p$ is a trivial family and the last equality is a consequence of the family $p\circ Q$ being trivial and equal to $p$ in a neighborhood of 1.
\end{proof}

\subsection{Applying the degeneration formula to $\shL_0$}
We are going to apply the degeneration formula Theorem~\ref{degen-formula} to the log smooth space $\shL_0$ over the standard log point which is the central fiber of the log smooth family $\shL\ra\AA^1$ from the previous subsection.
The following theorem will be proved in the next chapters.

\begin{theorem} \label{prop-vanishing-pushforward} 
Given $\Gamma\in\Omega(\shL_0)$, we have
$P_*G_*\Phi^*\Delta^! \prod_V[\oM_V]^\vir=0$ unless $\Gamma$ is the graph 
\xymatrix{
\overset{V_1}{\bullet}
\ar@{-}[r]^(-.1){}="a"^(1.1){}="b" \ar^e@{-} "a";"b"
&\overset{V_2}{\bullet}
} with $V_1$ an $X$-vertex and $V_2$ a $Y$-vertex and furthermore, $w_e=\beta\cdot D$, $\beta_{V_1}=\beta$, $n_{V_1}=\{1,...,n\}$, $n_{V_2}=\emptyset$ and $\beta_{V_2}$ is $w_e$ times the class of a fiber of $p:Y\ra D$.
\end{theorem}

For the remainder of this section, we deduce the main Theorem~\ref{mainthm-intro} from Theorem~\ref{prop-vanishing-pushforward}. 
Set $L_D:=D\times\AA^1$ which we view as the intersection of the components $L_X$ and $L_Y$ of $\shL_0$ and thus as a divisor in $L_X$ as well as in $L_Y$.
Set $d=\beta\cdot D$ and let $\Gamma$ in the following denote the exceptional graph given in Theorem~\ref{prop-vanishing-pushforward}.
In light of Lemma~\ref{lemma-Tot-is-p-pushforward}, we conclude from 
Theorem~\ref{degen-formula} and Theorem~\ref{prop-vanishing-pushforward} that 
\begin{equation} \label{class-pulled-thru-degenform}
\begin{array}{l}
[\oM(\Tot(\shO_X(-D)),\beta)]^\vir \\
\ = d\cdot P_* G_*\Phi^*\Delta^!\left([\oM_{(d)}(L_X(\log L_D),\beta_{V_1})]^\vir\times [\oM_{(d)}(L_Y(\log L_D),\beta_{V_2})]^\vir\right).
\end{array}
\end{equation}
Note that $\bigodot _V \oM_V$ can be identified with the top left corner in the diagram\footnote{Despite the notation, $\oM_{(d)}(L_X(\log L_D),\beta_{V_1})$ isn't proper but this doesn't affect us.} 
$$
\xymatrix{
\oM_{(d)}(L_X(\log L_D),\beta)\times_{L_D} \oM_{(d)}(L_Y(\log L_D),\beta_{V_2})\ar_{\pr_1}[d]&\ar_(.13){\Phi}[l] \oM_\Gamma\ar^(.4)G[r]&\oM(\shL_0,\beta)\ar^P[d]\\
\oM_{(d)}(X(\log D),\beta)\ar^F[rr]&&\oM(X,\beta)
}
$$
and this diagram is commutative because the curves in 
$$\oM_{(d)}(L_Y(\log L_D),\beta_{V_2})=\oM_{(d)}(Y(\log D),\beta_{V_2})$$ 
become entirely unstable when composing with $p:Y\ra D$.
Therefore, the right hand side in \eqref{class-pulled-thru-degenform} equals
$$ d\cdot F_*(\pr_1)_*\Phi_*\Phi^*\Delta^!\left([\oM_{(d)}(L_X(\log L_D),\beta)]^\vir\times [\oM_{(d)}(L_Y(\log L_D),\beta_{V_2})]^\vir\right) $$
$$=d\deg(\Phi)\cdot  F_*(\pr_1)_*\Delta^!\left([\oM_{(d)}(L_X(\log L_D),\beta)]^\vir\times [\oM_{(d)}(L_Y(\log L_D),\beta_{V_2})]^\vir\right). $$
We note that $\deg(\Phi)=1$ by \cite[Equation (1.4)]{KLR17}.
So, in order to identify the last equation with the left hand side in Theorem~\ref{mainthm-intro} (up to moving the factor $(-1)^{d+1}d$ to the other side) and thereby via \eqref{class-pulled-thru-degenform} deducing the Theorem, it suffices to observe that $[\oM_{(d)}(L_X(\log L_D))]^\vir$ Gysin-restricts to $[\oM_{(d)}(X(\log D))]^\vir$ when confining the evaluation to be in $D\times\{0\}$ and then to prove the following result.
\begin{proposition} 
Under the evaluation map $\ev:\oM_{(d)}(L_Y(\log L_D),\beta_{V_2})\ra D$, we find $$\ev_*[\oM_{(d)}(L_Y(\log L_D),\beta_{V_2})]^\vir = \frac{(-1)^{d+1}}{d^2}[D].$$
\end{proposition}
\begin{proof} 
The maps $L_Y(\log L_D)\ra Y(\log D_0)\stackrel{p}\lra D\ra \pt$ are log smooth and $\beta_{V_2}$ is the $d$'th multiple of a fiber class of $p$, 
so we find 
$$[\oM_{(d)}(L_Y(\log L_D)/\pt,\beta_{V_2})]^\vir=[\oM_{(d)}(L_Y(\log L_D)/D,\beta_{V_2})]^\vir.$$
Since $\vdim \oM_{(d)}(Y(\log D_0)/D,\beta_{V_2})=\dim D$ (e.g. by inserting $h=0$ in \cite[\S6.3]{BP08}), necessarily $\ev_*[\oM_{(d)}(L_Y(\log L_D),\beta_{V_2})]^\vir$ is a multiple of $[D]$. We can compute the degree by Gysin pulling back to a point in $D$ and the formation of the virtual fundamental class is compatible with this pullback by \cite[Prop. 7.3]{BF97}.
Hence, it remains to show that 
$$\deg \big([\oM_{(d)}(\Tot(\shO_{\PP^1}(-1))(\log (\{0\}\times\AA^1)),d[\PP^1])]^\vir\big)=\frac{(-1)^{d+1}}{d^2}.$$
The exact sequence
$$ 0 \ra \shT_{\PP^1(\log\{0\})}\ra \shT_{\Tot\big(\shO_{\PP^1}(-1)\big)(\log (\{0\}\times\AA^1))}|_{\PP^1}\ra \shO_{\PP^1}(-1) \ra 0$$
relates the local part of this moduli space to the twist by the obstruction bundle, hence we need to show that
\begin{equation} \label{BP-result-eqn}
\deg \big(e(O)\cap[\oM_{1}(\PP^1(\log \{0\}),d[\PP^1])]^\vir\big)=\frac{(-1)^{d+1}}{d^2}
\end{equation}
where $O=R^1\pi_*f^*\shO_{\PP^1}(-1)$ for $\oM_{1}(\PP^1(\log \{0\}))\stackrel{\pi}{\lla}\shC\stackrel{f}{\lra}\PP^1$ the maps from the universal curve to moduli space and target.
By \cite[Lemma~6.3]{BP08}\footnote{This is also Theorem 5.1 of \cite{BP05}.}, 
specializing the equivariant parameter $t_2$ in loc.cit. to $1$, the left hand side of \eqref{BP-result-eqn} equals the coefficient of $1/u$ in 
$$
\GW(0|-1,0)_{(d)} = \frac{(-1)^{d+1}}{d}\left(2\sin\frac{du}{2}\right)^{-1}
$$
which is readily seen to be $\frac{(-1)^{d+1}}{d^2}$.
\end{proof}

\section{Excluding multiple gluing points of curves in $L_X$}
\label{sec-exclude-mult-glue-X}
In this section, we prove the statement of Theorem~\ref{prop-vanishing-pushforward} for graphs $\Gamma\in\Omega(\shL_0)$ that have a vertex with bipartition membership in $X$ that has at least two adjacent edges.
\begin{lemma} \label{lemma-no-mult-edges-X}
Let $\Gamma\in\Omega(\shL_0)$ be a graph with an $X$-vertex $V$ with $r>1$ adjacent edges, then $[\oM_\Gamma]^\vir=0$.
\end{lemma}
\begin{proof} Let $r+s$ be the number of edges of $\Gamma$.  Since maps from compact curves to $\AA^1$ are constant, the evaluation map $\oM_V \to (D\times\AA^1)^r$ factors through $D^r \times \AA^1$ where $\AA^1$ is embedded diagonally in $\AA^r$.  The same is true for vertices corresponding to components in $L_Y$, since the bundle $\shO_Y(D_\infty)$ is nef. Using this,
we rewrite the diagram \eqref{eq-glue-square} by separating out the factors for $V$  to find Cartesian squares.
\[
\xymatrix@C=30pt
{
\oM_V\times_{L_D^r} \bigodot_{V'\neq V}\oM_{V'} \ar[r] \ar^{\ev}[d] &  \oM_V\times \prod_{V'\neq V}\oM_{V'} \ar^{\ev}[d]\\
(D^r\times\AA^1)\times (D\times\AA^1)^{s} \ar^{(\id\times\diag)\times\id=:\delta}[d]\ar^{\Delta'}[r] & (D^r\times\AA^1)^2\times (D\times\AA^1)^{2s}\ar^{(\id\times\diag)\times\id}[d]\\
(D\times\AA^1)^{r}\times (D\times\AA^1)^{s}\ar^{\Delta}[r]& (D\times\AA^1)^{2r}\times (D\times\AA^1)^{2s}.
}
\]
Let $N$ denote the normal bundle of the embedding $\Delta$ which has rank $(r+s)(\dim D+1)$ and $N'$ denote that of $\Delta'$ which has rank $r\dim D+1+s(\dim D+1)$.
Set $E=(\delta^*N)/N'$ which is of rank $r-1$, let $c_{r-1}(E)$ be its top Chern class.
For any $k$ and $\alpha\in A_k\big(\oM_V\times \prod_{V'\neq V}\oM_{V'}\big)$, the excess intersection formula says
$$\Delta^!\alpha= c_{r-1}(E)\cap(\Delta')^!\alpha.$$
Note that the normal bundle of the bottom right vertical map is trivial and, by Cartesianness of the lower square, its pullback under $\Delta'$ is isomorphic to $E$, so $c_{r-1}(E)=0$ because $r>1$ by assumption.
Applying this to the virtual fundamental class $\alpha=[\oM_V]^\vir\times \prod_{V'\neq V}[\oM_{V'}]^\vir$ proves the Lemma.
\end{proof}


\section{Comparing stable maps to $Y(\log D_0)$ and $D$}
We let $D$ be an arbitrary smooth projective variety (with the trivial log structure), and $Z$ a log scheme\footnote{We assume the log structure is in the Zariski topology to satisfy the assumptions of \cite{GS13}.}  with a log smooth and projective morphism $p:Z\ra D$.  
This induces a morphism of spaces of (log) stable maps.  The example relevant to our main theorem is given by $Z=Y(\log D_0)$ as in Section \ref{section-setup}.  In this case $\shT_{Z/D} = \shO_Y(D_\infty)$ which is a nef line bundle, since $D_\infty$ is an effective divisor with nef normal bundle.
This implies that the relative log tangent bundle has no higher cohomology when pulled back under any morphism $\PP^1 \to Z$, or more generally any morphism from a genus zero curve.  This property is useful in studying the induced morphism on spaces of stable maps.
In this section it is not necessary that the underlying scheme of $Z$ is smooth and we can state our main result in arbitrary genus (although most examples will be in genus zero).  Consequently, our notation for the space of log maps will just be $\oM_{g,n}(Z,\beta)$ where we do not try to describe the type of contact at the points.  This will just be the disjoint union over all possible conditions to impose at the markings. 

Let $\htz$ denote the submonoid of $H_2(Z,\ZZ)$ spanned by effective curve classes (including zero).  Fixing a class\ $\beta\in \htz$ we would like to compare the virtual fundamental classes of $\oM_{g,n}(Z, \beta)$ and $\oM_{g,n}(D,p_*\beta)$.  The natural map between them is induced by forgetting the log structures, composing the maps, and stabilizing.  We want to factor that map via an intermediate space in order to deal with these steps separately. To this end we need to discuss certain stacks of curves with extra structure. 

First, we have $\foM_{g,n}$ which parametrizes prestable curves of genus $g$ with $n$ markings.  The stack $\oM_{g,n}(D,p_*\beta)$ has a natural map to $\foM_{g,n}$, given by remembering the underlying curve, but forgetting the map to $D$.  We would like to forget the map to $D$, but remember the homology class of each irreducible component of the source curve, so we want to factor this map through a stack $\mathfrak M_{g,n,\htd}$ which parametrizes prestable genus $g$ curves together with an effective curve class on each irreducible component and satisfying the stability condition that components of degree and genus zero must contain three special points. 
Families of such curves are required to satisfy the obvious continuity condition that when a component degenerates, the sum of the classes on the degenerations is equal to the class of the original component. This stack was introduced by Costello in \cite{costello} and used for a similar
purpose by Manolache in \cite{Ma12b}.
The crucial fact for us will be that the deformation theory of such a decorated curve is identical to that of the undecorated curve, i.e.  $\mathfrak M_{g,n,\htd}$ is \'etale over $\mathfrak M_{g,n}$, see \cite[Proposition 2.0.2]{costello}.  Therefore, instead of thinking of the standard obstruction theory on $\oM_{g,n}(D,p_*\beta)$ as being a relative obstruction theory over $\foM_{g,n}$, we can think of it as a relative obstruction theory over $\foM_{g,n,\htd}$.  Similarly, we can factor the structure map from $\oM_{g,n}(Z,\beta)$ to $\foM_{g,n}^{\log}$ (the stack parametrizing log smooth families of curves over log schemes) through a scheme $\foM^{\log}_{g,n,\htz}$ which again records an effective curve class on each irreducible component.
We arrive at the following commutative diagram, where $\shM$ is taken to make the right hand square Cartesian.  (We suppress some indices in the subscripts for clarity both here and in what follows.)

\begin{equation} \label{main-diagram-Y-D}
\begin{aligned}
\xymatrix@C=30pt
{
\oM(Z,\beta)\ar^-u[r]\ar[d]& \shM\ar^-v[r]\ar[d]& \oM(D,p_*\beta)\ar[d]\\
\foM^{\log}_{\htz} \ar^{id}[r]& \foM^{\log}_{\htz} \ar^\nu[r]& \foM_{\htd}.
}
\end{aligned}
\end{equation}

The main reason for introducing the labeled curves, is that the stabilization map is now defined on the bottom row, since we can tell which components become unstable just from the discrete data of the homology classes.  (We remark that the existence of the relevant stabilization maps depends crucially on the fact that 0 is indecomposable in $\htd$.)
The mapping $\nu$ is given by forgetting the log structures, applying $p_*$ to the homology markings, and then stabilizing as necessary, so we have a natural morphism from the universal curve $\foC^{\rm log}_{\htz} $ to the pullback under $\nu$ of the universal curve $\foC_{\htd}$ over $\foM_{\htd}$.  This map contracts rational curves that are destabilized by forgetting the log structures and pushing forward the homology class.  In particular, any positive-dimensional fiber of this map is a union of $\PP^1$'s whose labelling becomes zero in $\htd$.

Now we are in a position to state the main theorem of this section.

\begin{theorem}\label{thm-pullback}
Let $p:Z \to D$ be a log smooth morphism where $D$ has trivial log structure.  Suppose that for every log stable morphism $f:C \to Z$ of genus $g$ and class $\beta$ we have $H^1(C,f^*\shT_{Z/D}) = 0$, then
  $$[\oM_{g,n}(Z,\beta)]^\vir = u^* \nu^! [\oM_{g,n}(D,p_*\beta)]^\vir$$
provided that $\oM_{g,n}(D,p_*\beta) \neq \emptyset$.
  In particular, if $[\oM_{g,n}(D,p_*\beta)]^\vir$ can be represented by a cycle supported on some locus $W\subset \oM_{g,n}(D,p_*\beta)$, then $[\oM_{g,n}(Z,\beta)]^\vir$ can be represented by a cycle supported on $w^{-1}(W)$ where $w = v\circ u$ is the natural map between these stacks.
\end{theorem}
  
\begin{remark}
Note that while the last statement involves only the standard spaces of (log) stable maps, we do not know how to formulate the precise relationship without passing through $\shM$.   Since $w$ does not seem to be flat in general, it is not apparent how to pull back classes under $w$.  \end{remark}

To prove this result, the point will be to show that $\shM$ itself has a relative perfect obstruction theory such that the associated virtual class $[\shM]^\vir$ satisfies the equations
\begin{equation}\label{first} \nu^![\oM_{g,n}(D,p_*\beta)]^\vir = [\shM]^\vir,
\end{equation}
 \begin{equation}\label{second} u^*[\shM]^\vir = [\oM_{g,n}(Z,\beta)]^\vir.
 \end{equation}
To obtain a relative perfect obstruction theory on $\shM$ we can just pull back the obstruction theory using $\nu$.  The fact that Equation \ref{first} holds follows from the fundamental base change property of virtual fundamental classes which we will recall here for the reader's convenience.

\begin{proposition}
Assume that we are given a fiber diagram of Artin stacks
\begin{equation} 
\begin{aligned}
\xymatrix@C=30pt
{
\shM\ar[r]^m\ar[d]& \shN\ar[d]^f\\
 G \ar^\mu[r]& H.
}
\end{aligned}
\end{equation}
with $G$ and $H$ pure dimensional, $f$ of DM type, and so that $\shM$ admits a stratification by quotient stacks.
If there is a relative perfect obstruction theory with virtual tangent bundle $\mathfrak E$ for $\shN$ over $H$, then there is an induced relative perfect obstruction theory with virtual tangent bundle $m^*{\mathfrak E}$ for $\shM$ over $G$, and the associated virtual fundamental classes satisfy $\mu^!([\shN]^\vir) = [\shM]^\vir$ if $\mu$ is either flat or an l.c.i. morphism.  
\end{proposition}

\begin{proof}
This all follows immediately from results in \cite{Ma12a} using the definition that 
$[\shN]^\vir = f_{\mathfrak E}^!([H])$.  The case of $\mu$ flat is a special case of Theorem 4.1 (ii) and the case where $\mu$ is lci is a special case of Theorem 4.3.  
\end{proof}

  To apply the Proposition to our situation, we factor the morphism
$\nu$ as the composition of the graph followed by the projection and use the fact that the latter is flat and the former is lci (since  $\foM_{\htd}$ is smooth).  This is also how we define $\nu^!$ on Chow groups.

To obtain Equation \ref{second} we want to use that this pulled back obstruction theory has a geometric interpretation, which we describe now.
The obstruction theory on $\oM_{g,n}(D,p_*\beta)$ arises from the fact that it is an open subset of the relative Hom stack $\Hom(\foC_{g,n}/\foM_{g,n},D)$.  In keeping with our diagram above, we want to consider it instead as an open substack of $\Hom (\foC_{\htd}/\foM_{\htd}, D)$.  For what follows, it will be convenient to notice that it is clearly contained in the open set where the labelling of the components of the fibers of the universal curve coming from the universal property of $\foM_{\htd}$ agrees with the labelling coming from the homology of the image under the morphism.  We denote this open subset by $\Hom^0(\foC_\htd / \foM_\htd , D)$.  This stack parametrizes morphisms from homology labeled curves satisfying the condition that the pushforward of the homology class of a component is given by the label of that component.  Since formation of the relative Hom scheme is compatible with base change, we know that $\shM$ is an open subset of $\Hom(\nu^{-1}\foC_\htd / \foM^{\rm log}_\htz, D)$, and the obstruction theory obtained by pullback under $\nu$ is simply the natural obstruction theory for this Hom scheme.

The key observation for proving Equation \ref{second} is that $\shM$ can also be thought of as an open subset of $\Hom(\foC^{\rm log}_\htz / \foM^{\rm log}_\htz, D)$.  Inside this space there is an analogous open $\Hom^0$ where we demand that the labelling obtained from the morphism to $D$ agrees with $p_*$ of the labelling coming from the universal property.

\begin{lemma}  The natural morphism 
$$\Psi:\Hom^0(\nu^{-1}\foC_\htd / \foM_\htz^{\rm log} , D) \to \Hom^0(\foC^{\rm log}_\htz / \foM^{\rm log}_\htz , D) $$ induced by the morphism $f: \foC^{\rm log}_\htz \to  \nu^{-1}\foC_\htd $ is an isomorphism.

\end{lemma}
\begin{proof}
The map $f$ just contracts some destabilized $\PP^1$'s.  The superscripts 0 imply that the rational curves contracted by $f$ are necessarily contracted by every morphism parametrized by the $\Hom^0$ schemes we are considering. 
Given an $S$-valued point of $\foM_\htz^{\rm log}$ corresponding to a marked log curve with underlying nodal curve $C \to S$, denote by $C' \to S$ the family obtained by applying $\nu$ and $c:C \to C'$ the associated contraction mapping (which corresponds to $f$).  A morphism $g: C' \to D$ corresponding to a point of $\Hom^0(\nu^{-1}\foC_\htd / \foM_\htz^{\rm log} , D)$ is taken by $\Psi$ to $\Psi(g) = g\circ c : C \to D$.  The statement that this morphism $\Psi$ gives an isomorphism amounts to the statement that any morphism $\tilde g : C \to D$ which is constant on those components of $C$ which are contracted by $c$ factors uniquely through $c$.  This is obvious when $S$ is the spectrum of an algebraically closed field.  To see that this holds over an arbitrary base, one needs to use the standard fact about contraction maps between families of curves that $c_*\shO_C = \shO_{C'}$.  The result then follows from Lemma 2.2 of \cite{BM96}.
\end{proof}

In addition to being isomorphic stacks, the natural obstruction theories on $\shM$ induced by these two descriptions agree, because the contracted genus zero components make no contribution to the cohomology of $f^*(\shT_D)$.
To prove Equation \ref{second}, we note that we now have that $\oM(Z,\beta)$ and $\shM$ are two stacks with relative perfect obstruction theories over the same base, $\foM^{\rm log}_\htz$.  They are both open subsets of stacks of morphisms from the same family of curves. In the case of the space of log stable maps to $Z$, we have that 
$\oM_{g,n}(Z,\beta)$ is an open subset of the logarithmic $\Hom$ stack $\Hom^{\rm log}(\foC^{\rm log}_{g,n}/\foM^{\rm log}_{g,n}, Z)$ parametrizing logarithmic maps from fibers 
of the universal family to $Z$.  The obstruction theory for this scheme is given by the cohomology of the pullback of the logarithmic tangent bundle of $Z$.  In particular, this obstruction complex depends only on the underlying morphism of schemes, not on the log structures.  The obstruction theory of $\shM$ comes from the cohomology of $f^*(\shT_D)$. In order to compare them, we use the short exact sequence
$$0\to \shT_{Z/D} \to \shT_Z\to \shT_D \to 0.$$

Given that we have  $H^1(C,f^*\shT_{Z/D})=0$ for all stable maps $f:C \to Z$, it follows from the associated long exact sequence in cohomology that $u$ is smooth with relative cotangent complex supported in one term, given by $H^0(f^*\shT_{Z/D})$ and we have a very special case of a compatibility datum, implying via \cite{Ma12a} Corollary 4.9 that $[\oM(Z)]^\vir = u^*[\shM]^\vir$ where $u^*$ denotes smooth pullback.

\qed


\section{Excluding nontrivial curves in $L_Y$}
For given $\Gamma\in\Omega(\shL_0)$, recall the maps $G$ and $P$ from \S\ref{sec-degen-formula}. Let $r_V$ denote the number of edges of $\Gamma_V$, i.e. edges in $\Gamma$ adjacent to $V$. In the following, we will use abusive notation by referring by $n_V$ not only to the set introduced in \S\ref{sec-degenformula} but also its size.
Note that $P\circ G:\oM_\Gamma\ra \oM_{0,n}(X,\beta)$ factors through the \'etale map $\Phi$, say $P\circ G=P'\circ\Phi$ for $P':\bigodot_V\oM_V\ra \oM_{0,n}(X,\beta)$.
Now $P'$ is induced by maps $P_V:\oM_V\ra\oM_{0,n_V+r_V}(X,\tilde p_*\beta_V)$ on the factors of $\bigodot_V\oM_V$ which is on each factor the composition of a stable map with $\tilde p$ which of course only has an effect for $\oM_V$ with $V$ associated to $Y$. For a $Y$-vertex $V$, recall the forgetful morphism $w:\oM_{\Gamma_V}(Y(\log D_0),\beta_V)\ra\oM_{0,n_V+r_V}(D,p_*\beta_V)$ that we studied in the previous section.

\begin{lemma} \label{lem-zero-pushforward}
If $\beta_V \in H^+_2(Y)=H_2^+(L_Y)$ is a curve class such that $p_*\beta_V \in H_2^+(D)$ is non-zero, then the class $[\oM_{\Gamma_V}(L_Y(\log L_D),\gamma)]^\vir$ pushes forward to zero under the natural map to $\oM_{0,n_V+r_V}(D,p_*\gamma)$.
\end{lemma}

\begin{proof} 
As we said in the previous section, $\shT_{Y(\log D_0)/D}\cong\shO_Y(D_\infty)$. 
Thus, for any genus zero stable map $f:C\ra Y(\log D_0)$ of type $\Gamma_V$, we have 
$\rk H^0(C,f^*\shO_Y(D_\infty))=\beta_V\cdot D_\infty+1$ and this number is the relative virtual and actual dimension of the map $u$ in \eqref{main-diagram-Y-D}.
This coincides with the relative dimension of $w$ because the map that forgets the log structure is of relative dimension zero.
Denoting by $\oM_{\Gamma_V}(L_Y(\log L_D),\beta_V)\stackrel{\pi}{\la}\shC\stackrel{f}{\ra} Y(\log D_0)$ the maps from the universal curve and setting $\shE=R^1\pi_*f^*\shO_Y(-D_\infty)$, we also have that
$$ [\oM_{\Gamma_V}(L_Y(\log L_D),\beta_V)]^\vir = e(\shE) \cap [\oM_{\Gamma_V}(Y(\log D_0),\beta_V)]^\vir.$$
Since the rank of $\shE$ is $\beta_V \cdot D_\infty - 1$, we find the virtual dimension of $\oM_{\Gamma_V}(L_Y(\log L_D),\beta_V)$ to be strictly greater than the virtual dimension of 
$\oM_{0,n_V+r_V}(D,p_*\beta_V)$. But Theorem~\ref{thm-pullback} implies that the pushforward of $[\oM_{\Gamma_V}(Y(\log D_0),\beta_V)]^\vir$ can be supported on any cycle that supports 
$[\oM_{0,n_V+r_V}(D,p_*\beta_V)]^\vir$.  Since $\oM_{0,n_V+r_V}(D,p_*\beta_V)$ is a Deligne-Mumford stack and we are working in Chow groups with rational coefficients, this implies the desired vanishing.
\end{proof}

\begin{remark}
It was pointed out to us by Feng Qu that a similar result to Lemma~\ref{lem-zero-pushforward}  was proved in \cite[Prop. 3.2.2.]{LLQW16} which would imply this result when $\shE = 0$. 
\end{remark}

\begin{proposition} \label{prop-nontrivial-in-Y-pushes-to-zero}
If $\Gamma\in \Omega(\shL_0)$ has a vertex $V$ with bipartition membership in $Y$ and $\beta_V$ is not a multiple of the class of a fiber of $p$, then 
$P_*G_*\Phi^*\Delta^!\prod_V[\oM_V]^\vir=0$.
\end{proposition}

\begin{proof} 
Let $\Gamma$ and $V$ be as in the assertion. Set $r:=r_V$.  The argument here is simple, but we include 
a large diagram to remind the reader of the names of the maps.  The obvious, but important, point
is that the evaluation maps from $\oM_V = \oM_{\Gamma_V}(L_Y(\log L_D),\beta_V)$ to $L_D$ factors through
the map $P_V$ to $\oM_{0,n_V+r}(L_D, p_*\beta_V)$.  The stack $M'$ below is defined to make the bottom right square Cartesian.

\[
\xymatrix@C=30pt
{
\oM_\Gamma \ar[d]_G\ar[r]^(.3){\Phi}& \oM_V\times_{L_D^r} \bigodot_{V'\neq V}\oM_{V'} \ar[r] \ar[d] &  \oM_V\times \prod_{V'\neq V}\oM_{V'} \ar^{P_V\times id}[d]\\
\oM(X\cup Y)\ar[d]_P & M' \ar[ld]_\tau \ar[d]\ar[r] & \oM_{0,n_V+r}(L_D,p_*\beta_V) \times \prod_{V'\neq V}\oM_{V'} \ar^{\ev}[d]\\
\oM(X)&L_D^r\ar^{\Delta}[r]&( L_D)^{2r}.
}
\]
The only point not yet explained is the existence of the map $\tau$ but this is just the usual clutching construction for boundary strata.
Since $p_*\beta_V$ is nonzero, Lemma~\ref{lem-zero-pushforward} applies and we have
$$P_{V*}[\oM_{\Gamma_V}(L_Y(\log L_D),\beta_V)]^\vir=0.$$ 
Rewriting the class in the proposition
as 
$$\deg(\Phi) \cdot \tau_*\Delta^! \big(P_{V*}[\oM_{\Gamma_V}(L_Y(\log L_D),\beta_V)]^\vir \times 
\prod_{V'\neq V}[\oM_{V'}]^\vir\big)$$ 
the result follows immediately.
\end{proof}

\begin{proof}[Proof of Theorem~\ref{prop-vanishing-pushforward}]
Let us collect what is implied for the graph $\Gamma\in\Omega(\shL_0)$ if the pushforward to $\oM_{0,n}(X,\beta)$ of the corresponding virtual fundamental class is non-trivial:
by Lemma~\ref{lemma-no-mult-edges-X}, the X-vertices of the graph have no more than one adjacent edge and by Proposition~\ref{prop-nontrivial-in-Y-pushes-to-zero}, every curve component in $Y$ needs to be a multiple of a fiber. 
If a $P_*G_*[\oM_\Gamma]^\vir$ is non-trivial, it already follows that $\Gamma$ has a unique $Y$-vertex $V$ and we are only left with showing that this has only a single adjacent edge to it.
If $\beta_V$ is a multiple of a fiber class, borrowing notation from the proof of Proposition~\ref{prop-nontrivial-in-Y-pushes-to-zero} and setting
$\oM^\circ:=\prod_{V'\neq V}[\oM_{V'}]^\vir$,
the map from
$\oM_{\Gamma_V}(L_Y(\log L_D),\beta_V) \times_{(D\times\AA^1)^{r_V}} \oM^\circ$
to $\oM_{0,n}(X)$ factors through
$D \times_{(D \times \AA^1)^{r_V}}\oM^\circ$.  
Hence, it suffices to check that the pushforward of the virtual class from $\oM_{\Gamma_V}(L_Y(\log L_D),\beta_V)$ to $D$ vanishes which is precisely what the next lemma achieves.
\end{proof}

\begin{lemma}
If $\beta_V$ is a multiple of the class of a fiber of $p:Y\ra D$, then the pushforward of $[\oM_{\Gamma_V}(L_Y(\log L_D),\beta_V)]^\vir$ under the evaluation map to $(D \times \AA^1)^{r_V}$ is trivial if $n_V+r_V>1$.
\end{lemma}

\begin{proof} Set $r:=r_V$.
The evaluation map from $\oM_{\Gamma_V}(L_Y(\log L_D),\beta_V)$ to $D^r\times \AA^r$ factors through the embedding $D \to D^r\times \AA^r$ given by $\diag \times \{0\}$.
However, the virtual dimension of $\oM_{\Gamma_V}(L_Y(\log L_D),\beta_V)$ is $\dim(D) + n_V+r-1$, so for $n_V+r>1$ this vanishing is immediate.
\end{proof}



\end{document}